\documentclass[12pt,leqno]{amsart}
\newtheorem{theorem}{Theorem}[section]
\newtheorem{definition}{Definition}[section]

\newtheorem{remark}{Remark}[section]

\newtheorem{example}{Example}[section]

\newcommand{\be}{\begin{equation}}
\newcommand{\ee}{\end{equation}}
\numberwithin{equation}{section}
\newcommand{\bea}{\begin{eqnarray}}
\newcommand{\eea}{\end{eqnarray}}
\newcommand{\beb}{\begin{eqnarray*}}
\newcommand{\eeb}{\end{eqnarray*}}
\usepackage{amssymb,amsfonts,amsthm,setspace,indentfirst}
\usepackage[dvips]{graphics}
\usepackage{epsfig}
\begin{document}
\title{Rough ideal convergence in a partial metric space}
\author{Sukila Khatun$^{1}$, Amar Kumar Banerjee$^{2}$ and Rahul Mondal$^{3}$}
\address{$^{1}$,$^{2}$ Department of Mathematics, The University of Burdwan,
Golapbag, Burdwan-713104, West Bengal, India.} 
\address{$^{3}$Department of Mathematics, Vivekananda Satavarshiki Mahavidyalaya, Manikpara, Jhargram -721513, West Bengal, India.}
\email{$^{1}$sukila610@gmail.com}
\email{$^{2}$akbanerjee@math.buruniv.ac.in, akbanerjee1971@gmail.com}
\email{$^{3}$imondalrahul@gmail.com}
\begin{abstract}
In this paper, using the concept of ideal, we study the idea of rough ideal convergence of sequences which is an extension of the notion of rough convergence of sequences in a partial metric space. We define the set of rough $\mathcal{I}$-limit points and the set of rough $\mathcal{I}$-cluster points and then we prove some relevant results associated with these sets.  
\end{abstract}
\maketitle
\section{\bf{Introduction }}
The notion of statistical convergence \cite{HF, {HS}} occupies a prominent role in the literature of summability theory as a generalization of ordinary convergence. Since then many types of generalizations and applications of statistical convergence are being carried out. The notion of ideal convergence \cite{PK3} appeared as a generalizations of both ordinary and statistical convergence got an immense importance in least two decades when topologists turned their attention on generalization of convergence in a more sophisticated way. Thereafter lot of works by several authors \cite{ {AKBPAL}, {KD}, {PD}, {PK2}, {BK}} were carried out in this direction. The concept of $\mathcal{I}$-convergence was developed to many areas by several researchers \cite{{KD}, {PD}, {PK2}}. \\
H. X. Phu \cite{{PHU}, {PHU1}} introduced the notion of rough convergence in a finite dimensional normed linear space and also in infinite dimensional normed spaces. The notion of rough convergence has interesting applications that occurs naturally in numerical analysis. The idea of rough convergence extended into rough statistical convergence using the notion of natural density by Ayter \cite{AYTER1}. Again, the idea of rough statistical convergence extended into rough ideal convergence using the concept of ideals by S. K. Pal, D. Chandra and S. Dutta \cite{PAL}.
Several works were done by many authors \cite{{RMROUGH}, {DR}, {NH}, {PMROUGH1}, {SUK1}} in many generalized spaces. \\
In 1994, the notion of partial metric space was introduced by S. Matthews \cite{MATW}, which is a generalization of a metric space. Recently many works \cite{ {SUK2}, {SUK4}, {DUN}, {FN}} have been carried out in partial metric spaces. \\
In our present study, we discuss the idea of rough ideal convergence in a partial metric space. We define the set of rough ideal limit points and prove several properties like boundedness, closedness etc of this set. We also prove several results associated with this set.

\section{\bf{Preliminaries}}

\begin{definition}\cite{BMW}
A partial metric on a non-empty set $X$ is a function $p: X\times X \longrightarrow [0, \infty)$ such that for all $x,y,z \in X$:\\
$(p1)$ $0 \leq p(x,x) \leq p(x,y)$ (nonnegativity and small self-distances),\\
$(p2)$ $x=y \Longleftrightarrow p(x,x)=p(x,y)=p(y,y)$ (indistancy both implies equality),\\
$(p3)$  $p(x,y)= p(y,x)$ (symmetry),\\
$(p4)$  $p(x,y) \leq p(x,z) + p(z,y) - p(z,z)$ (triangularity).\\
Then the pair $(X,p)$ is said to be a partial metric space.

Properties and examples of partial metric spaces were widely discussed in \cite{BMW}.
\end{definition}

\begin{definition} \cite{BMW}
In a partial metric space $(X,p)$, for $r>0$ and $x \in X$ we define the open and closed ball of radius $r$ and center $x$ respectively as follows  :
\begin{center}
    $B^{p}_{r}(x)=\{ y \in X : p(x,y)<p(x,x)+r  \}$ \\
$\overline{B^{p}_{r}}(x)=\{ y \in X : p(x,y) \leq p(x,x)+r \}.$ 
\end{center}
\end{definition}

\begin{definition} \cite{BMW}
Let $(X,p)$ be a partial metric space. A subset $U$ of $X$ is said to be a bounded in $X$ if there exists a positive real number $M$ such that $sup$ $\{ p(x,y): x,y \in U\}< M$.
\end{definition}

\begin{definition} \cite{BMW}
Let $(X,p)$ be a partial metric space and $\{x_{n}\}$ be a sequence in $X$. Then $\{x_{n}\}$ is said to converge to $x \in X$ if and only if $lim_{n\to\infty}p(x_{n},x)=p(x,x)$; i.e., if for each $\epsilon > 0$ there exists $k \in \mathbb{N}$ such that 
$|p(x_{n},x)-p(x,x)|< \epsilon$ for all $n \geq k$.
\end{definition}

\begin{definition} \cite{PK1}
Let $X$ be a non-empty set. 
Then a family of sets $\mathcal{I} \subset 2^X$ is said to be an ideal if \\
$(i) \phi \in \mathcal{I}$ \\
$(ii) A,B \in \mathcal{I} \Rightarrow A \cup B \in \mathcal{I}$ \\
$(iii) A \in \mathcal{I}, B \subset A \Rightarrow B \in \mathcal{I}$. \\ 
$\mathcal{I}$ is called non-trivial ideal if $X \notin \mathcal{I}$ and $\mathcal{I} \neq \{\phi\}$. A non-trivial ideal $\mathcal{I}$ in $X$ is called admissible if ${x} \in \mathcal{I}$ for each $x \in X$. Clearly the family $\mathcal{F(I)}=\{A \subset X: X \setminus A \in \mathcal{I}\}$ is a filter on $X$ which is called the filter associated with $\mathcal{I}$.
\end{definition}

\begin{definition} \cite{PK1}
    Let $\mathcal{I}$ be a non-trivial ideal of $\mathbb N$. A sequence $\{x_n\}$ in $\mathbb R$ is said to be $\mathcal{I}$-convergent to $x$ if for every $\varepsilon>0$, the set $A(\varepsilon)=\{n \in \mathbb N: |x_n-x| \geq \varepsilon\} \in \mathcal{I}$. \\
    If $\{x_n\}$ is $\mathcal{I}$-convergent to $x$, then $x$ is called $\mathcal{I}$-limit of $\{x_n\}$and we write $\mathcal{I}-lim \ x_n=x$.
\end{definition}

\begin{definition} \cite{PK1}
    Let $\mathcal{I}$ be an admissible ideal in $\mathbb N$. A sequence $\{x_n\}$ of real numbers is said to be $\mathcal{I^*}$-convergent to $x$ (shortly $\mathcal{I^*}-lim \ x_n=x$) if there is a set $A \in \mathcal{I}$, such that for $M=\mathbb N \setminus A=\{m_1 < m_2 < ... \}$, we have
    $lim_{k\to\infty}x_{m_k}=x$.
\end{definition}

\begin{definition} \cite{DUN}
Let  $\mathcal{I}$ be a non-trivial admissible ideal of $\mathbb N$. A sequence $\{x_n\}$ in a partial metric space $(X,p)$ is said to be ideal convergent ($\mathcal{I}$-convergent) to $x \in X$ if for every $\varepsilon>0$, the set $A(\varepsilon)=\{n \in \mathbb N: |p(x_{n},x)-p(x,x)| \geq \varepsilon\} \in \mathcal{I}$ i.e., if $\mathcal{I}-lim_{n\to\infty}p(x_{n},x)=p(x,x)$.
\end{definition}

\begin{definition} \cite{SUK2}
Let $(X,p)$ be a partial metric space. A sequence $\{ x_{n} \}$ in $X$ is said to be rough convergent (or $r$-convergent) to $x$ of roughness degree $r$ for some non-negative real number $r$ if for every  $\epsilon > 0$ there exists a natural number $k$ such that $| p(x_{n}, x)-p(x,x) |<r + \epsilon $ holds for all $n \geq k$.
\end{definition}

\section{\bf{Rough ideal convergence in partial metric spaces}}

\begin{definition}
Let $(X,p)$ be a partial metric space and $\{x_n\}$ be a sequence in $X$. Then the sequence $\{x_n\}$ is said to be rough ideal convergent of roughness degree $r$ (or in short rough $\mathcal{I}$-convergent or $r-\mathcal{I}$ convergent) to $x$, for any $\varepsilon>0$, the set 
\begin{center}
    $A(\varepsilon)=\{n \in \mathbb N : |p(x_n,x)-p(x,x) | \geq r+\varepsilon\} \in \mathcal{I}$.
\end{center}
\end{definition}

This can be denoted by  $x_{n} \stackrel{r-\mathcal{I}}{\longrightarrow} x$ in $(X,p)$, where we call $r$ as roughness degree. If we take $r=0$, we obtained $\mathcal{I}$-convergence in the partial metric space $(X,p)$. If a sequence $\{x_n\}$ is rough $\mathcal{I}$-convergent to $x$, then $x$ is said to be a rough $\mathcal{I}$-limit point of $\{x_n\}$. The set of all rough $\mathcal{I}$-limit points of a sequence $\{x_n\}$ is said to be the rough $\mathcal{I}$-limit set which is denoted by $\mathcal{I}-LIM^{r}x_{n}= \left\{x \in X : x_{n} \stackrel{r-\mathcal{I}}{\longrightarrow} x \right\}$. So, for roughness degree $r>0$, the rough $\mathcal{I}$-limit of a sequence may not be unique, as $\mathcal{I}-LIM^{r}x_{n}$ may contain more than one point. \\

\begin{theorem}
    Every rough convergent sequence in a partial metric space $(X,p)$ is rough $\mathcal{I}$- convergent in $(X,p)$.
\end{theorem}
\begin{proof}
Let $\{\xi_n\}$ be a rough convergent sequence in $(X,p)$ and let rough convergent to $\xi$.
Then for any $\varepsilon>0$, there exists $m \in \mathbb N$ such that for all $n \geq m$, we have $ |p(\xi_n,\xi)-p(\xi,\xi)| < r+\varepsilon$.
Now, the set $K(\varepsilon)=\{n \in \mathbb N : |p(\xi_n,\xi)-p(\xi,\xi) | \geq r+\varepsilon\} \subset \{1,2,3,...(m-1) \}$.
Since $\mathcal{I}$ is an admissible ideal, $K(\varepsilon) \in \mathcal{I}$.
Hence $\{\xi_n\}$ is rough $\mathcal{I}$ convergent in $(X,p)$.
    
\end{proof}

\begin{remark}
    The converse of the above theorem may not be true. We explain this fact by using the following example.
\end{remark}

\begin{example}
  Let $X=\mathbb R^+$ and $a>1$ ba a fixed real number and let $p: X\times X \longrightarrow \mathbb R^+$ be defined by $p(x,y)=a^{max \{x,y \}}$ for all $x,y \in X$. Then $(X,p)$ is a partial metric space. Let us take $\mathcal{I}=\{A \subset \mathbb N: \delta(A)=0 \}$.
  Let us consider a sequence $\{\xi_{n}\}$ as follows: 
  \begin{equation*}
    \  \xi_{n}= \begin{cases}
          k, & \text {if $n=k^2$}, \\
          0, & \text {otherwise}.
                 \end{cases}
  \end{equation*}
Let $\varepsilon>0$ be given.
Then $A_0=\{n \in \mathbb N : |p(\xi_n,0)-p(0,0) | \geq 1+\varepsilon\} \subset \{1^2,2^2,3^2,...\}=Q$ (say).
Since $Q \in \mathcal{I}$, $A_0 \in \mathcal{I}$. 
Similarly, the set
$A_k=\{n \in \mathbb N : |p(\xi_n,k)-p(k,k) | \geq 1+\varepsilon\} \subset Q$.
 Since $Q \in \mathcal{I}$, $A_k \in \mathcal{I}$.
 Hence $\{\xi_{n}\}$ is rough $\mathcal{I}$-convergent to 0 and $k$ of roughness degree 1.\\
 But \begin{equation*}
     |p(\xi_n,0)-p(0,0)|= \begin{cases}
         |p(k,0)-p(0,0)|=|a^k-1| & \text{if $n=k^2$},\\
         |p(0,0)-p(0,0)|=|1-1|=0 & \text{otherwise}. 
     \end{cases}
 \end{equation*}
 When $n=k^2$, there does not exist any positive integer $N$ such that the condition $|p(\xi_n,0)-p(0,0)|< r+ \varepsilon$ for all $n \geq N$ holds, since $|a^k-1| \longrightarrow \infty$ as $k^2 \longrightarrow \infty$.\\
 Hence $\{\xi_{n}\}$ is not rough convergent to 0 of any roughness degree $r>0$.
\end{example}

\begin{definition} \cite{SUK2}
    The diameter of a set $B$ in a partial metric space $(X,p)$ is defined by 
    \begin{center}
    $diam(B)$ = $sup$ $\{ p(x,y) : x, y\in B\}$.
\end{center}
\end{definition}

\begin{theorem}
Let $(X,p)$ be a partial metric space and $p(x,x)=a$ for all $x$ in $X$, where $a$ is a fixed positive real number. Then for a sequence $\{x_{n}\}$, we have $diam(\mathcal{I}-LIM^rx_n) \leq (2r+2a)$. 
\end{theorem}

\begin{proof}
Let us assume that $diam(\mathcal{I}-LIM^{r} x_{n}) > 2r+2a$.
Then $\exists$ $y,z \in \mathcal{I}-LIM^{r} x_{n}$ such that $p(y,z)>2r+2a$.
Let $\varepsilon>0$ and choose $\varepsilon \in (0,\frac{p(y,z)}{2}-r-a)$. 
Since $y,z \in \mathcal{I}-LIM^{r} x_{n}$, then we have 
\begin{equation*}
   A(\varepsilon)= \{ n \in \mathbb N: |p(x_{n},y)-p(y,y)| \geq r+\varepsilon \} \in \mathcal{I} 
\end{equation*}
and
\begin{equation*}
 B(\varepsilon)= \{ n \in \mathbb N:|p(x_{n},z)-p(z,z)|\geq r+\varepsilon \} \in \mathcal{I}.   
\end{equation*}
Now, $C(\varepsilon)=\mathbb N \setminus ( A(\varepsilon) \cup B(\varepsilon) ) \in \mathcal{F(I)}$ and so $C(\varepsilon) \neq \phi$.
Let $n \in C(\varepsilon)$. Then we can write
 \begin{equation*}
       \begin{split}
p(y,z) &\leq p(y,x_{n})+p(x_{n},z)-p(x_{n},x_{n})\\
       &=\{p(x_{n},y)-p(y,y)\}+\{p(x_{n},z)-p(z,z)\}-p(x_{n},x_{n})+p(y,y)+p(z,z)\\
       &< 2(r+\varepsilon)-a+a+a\\
       &=2r+2\varepsilon+a\\
       &<2r+ p(y,z)-2r-2a+a\\
       &=p(y,z)-a , \ \text{which is a contradiction}.    
       \end{split}
   \end{equation*}    
Hence $diam(\mathcal{I}-LIM^{r} x_{n}) \leq 2r+2a$. 
\end{proof}

\begin{theorem}
    Let $\{x_{n}\}$ be a sequence $\mathcal{I}$-convergent to $x$ in a partial metric space $(X,p)$. Then $\{ y \in \overline{B^{p}_{r}}(x):p(x,x)=p(y,y)\} \subseteq \mathcal{I}-LIM^{r}x_{n}$.
\end{theorem}

\begin{proof}
 Let $\varepsilon >0$. 
Since the sequence $\{x_{n}\}$ $\mathcal{I}$-converges to $x$, the set $A(\varepsilon)=\{n \in \mathbb N: |p(x_{n},x)-p(x,x)| \geq \varepsilon \} \in \mathcal{I}$.
 Let $y \in \overline{B^{p}_{r}}(x)$ such that $p(x,x)=p(y,y)$. 
 Then $p(x,y) \leq p(x,x)+r$ and $p(x,x)=p(y,y)$.
 Let $M=\mathbb N \setminus A(\varepsilon) \in \mathcal{F(I)}$. 
 So $M \neq \phi$. If $n \in M$, then
 \begin{equation*}
    \begin{split}
        p(x_{n},y) &\leq p(x_{n},x)+p(x,y)-p(x,x)\\
                   & = \{p(x_{n},x)-p(x,x)\}+p(x,y)\\
                   & < \epsilon + \{p(x,x)+r\} \\
                   &= p(x,x)+(r+\epsilon) 
    \end{split}
\end{equation*}
 Therefore, 
\begin{equation*}
    \begin{split}
p(x_{n},y)-p(y,y) & <p(x,x)-p(y,y)+(r+\epsilon)\\
                  & =(r+\varepsilon), \ \ \text{since} \ p(x,x)= p(y,y).
    \end{split}
\end{equation*}
This implies that $ |p(x_{n},y)-p(y,y)| < p(x_{n},y)-p(y,y) < (r+\varepsilon) $.
So, $M \subset \{ n \in \mathbb N: |p(x_{n},y)-p(y,y)| < r+\varepsilon \}$.
Consequently, $\{ n \in \mathbb N: |p(x_{n},y)-p(y,y)| < r+\varepsilon \} \in \mathcal{F(I)}$. 
So, $\{ n \in \mathbb N: |p(x_{n},y)-p(y,y)| \geq r+\varepsilon \} \in \mathcal{I}$.
Therefore, $y \in \mathcal{I}-LIM^{r}x_{n}$.
Hence $\{ y \in \overline{B^{p}_{r}}(x):p(x,x)=p(y,y)\} \subseteq \mathcal{I}-LIM^{r}x_{n}$.
\end{proof}

We recall that from \cite{SUK2}, a partial metric space is first countable. The next result is a consequence of this fact.

\begin{theorem}
   Let $\{x_{n}\}$ be a $r$-$\mathcal{I}$ convergent sequence in a partial metric space $(X,p)$. Then $\mathcal{I}-LIM^{r}x_{n}$ is a closed set for any degree of roughness $r \geq 0$.
\end{theorem}

\begin{proof}
Let $\varepsilon>0$ be given and
let $ y \in \mathcal{I}-LIM^{r}x_{n}$.
So, let $\{y_{n}\}$ be a sequence in $\mathcal{I}-LIM^{r}x_{n}$ such that $\{y_{n}\} {\longrightarrow} y$.
Since $\{y_{n}\} {\longrightarrow} y$, there exists $k_\frac{\varepsilon}{2} \in \mathbb N$ such that $|p(y_n,y)-p(y,y)| < \frac{\varepsilon}{2}$ for all $n \geq k_\frac{\varepsilon}{2}$.
We can choose $k_0 \in \mathbb N$ such that $k_0 > k_\frac{\varepsilon}{2}$.
Then $|p(y_{k_0},y)-p(y,y)| < \frac{\varepsilon}{2}$.
Again, since $\{y_{n}\} \subset \mathcal{I}-LIM^{r}x_{n}$, then $y_{k_0} \in \mathcal{I}-LIM^{r}x_{n}$ and so the set $A=\{n \in \mathbb N: |p(x_n,y_{k_0})-p(y_{k_0},y_{k_0})| \geq r+ \frac{\varepsilon}{2}\} \in \mathcal{I}$.
Now, we show that $A^c \subset \{n \in \mathbb N:|p(x_n,y)-p(y,y)|<r+\varepsilon \}.$
Let $m \in A^c$. 
Then $|p(x_m,y_{k_0})-p(y_{k_0},y_{k_0})| < r+ \frac{\varepsilon}{2}$.
Therefore,
\begin{equation*}
    \begin{split}
p(x_{m},y)-p(y,y) & \leq p(x_{m},y_{k_0})+p(y_{k_0},y)-p(y_{k_0},y_{k_0})-p(y,y) \\
         & \leq | p(x_{m},y_{k_0})-p(y_{k_0},y_{k_0})+p(y_{k_0},y)-p(y,y)| \\
         & \leq |p(x_{m},y_{k_0})-p(y_{k_0},y_{k_0})| + |p(y_{k_0},y)-p(y,y)| \\
         & < (r+\frac{\varepsilon}{2}) + \frac{\varepsilon}{2} \\
         & = r+\varepsilon
    \end{split}
\end{equation*}
Hence $m \in \{ n \in \mathbb N: |p(x_{n},y)-p(y,y)|<r+\varepsilon \}$.
Since $A^c \in \mathcal{F(I)}$, so $\{ n \in \mathbb N: |p(x_{n},y)-p(y,y)|<r+\varepsilon \} \in \mathcal{F(I)}$.
Hence $\{ n \in \mathbb N: |p(x_{n},y)-p(y,y)| \geq r+\varepsilon \} \in \mathcal{I}$.
Therefore, $y \in \mathcal{I}-LIM^{r}x_{n}$. 
Hence the theorem is proved.
\end{proof} 

\begin{definition} (cf \cite{AKB})
   A sequence $\{x_{n}\}$ in a partial metric space $(X,p)$ is said to be $\mathcal{I}$-bounded if for any fixed $u \in X$ there exists a positive real number $M$ such that 
   \begin{center}
       $\{ n \in \mathbb N : p(x_n,u) \geq M \} \in \mathcal{I}$.
   \end{center} 
\end{definition}

\begin{theorem}
Let $(X,p)$ be a partial metric space and $p(x,x)=a$, $ \forall x \in X$, where $a$ is a fixed positive real number. Then a sequence $\{x_{n}\}$ in $(X,p)$ is $\mathcal{I}$-bounded if and only if there exists a non-negative real number $r$ such that $\mathcal{I}-LIM^{r}x_{n} \neq \phi$.     
\end{theorem}

\begin{proof}
Let $u \in X$ be a fixed element in $X$.
Since the sequence $\{x_{n}\}$ is $\mathcal{I}$-bounded, there exists a positive real number $M$ such that the set $A=\{ n \in \mathbb N : p(x_n,u) \geq M\} \in \mathcal{I}$.
Let $\varepsilon>0$ and let $r=M+a$.
Clearly, $A^c \subset \{ n \in \mathbb N:|p(x_n,u)-p(u,u) < r+\varepsilon\}$.
For, let $k \in A^c$.
Then $p(x_k,u) < M$.
So, 
$p(x_k,u)-p(u,u) \leq |p(x_k,u)-p(u,u)| \leq |p(x_k,u)|+|p(u,u)| < M+a=r <r+\varepsilon $, which implies that
 $k \in \{ n \in \mathbb N:|p(x_n,u)-p(u,u)| < r+\varepsilon \}$. 
Since $A \in \mathcal{I}$, $\{ n \in \mathbb N : |p(x_n,u)-p(u,u) \geq r+\varepsilon \} \in \mathcal{I}$.
Therefore, $ u \in \mathcal{I}-LIM^{r}x_{n} $ and so
$\mathcal{I}-LIM^{r}x_{n} \neq \phi$. \\

Conversely, suppose that $\mathcal{I}-LIM^{r}x_{n} \neq \phi$.
So, let $u \in \mathcal{I}-LIM^{r}x_{n}$ . 
Then for $\varepsilon >0$, the set
$A=\{ n \in \mathbb N : |p(x_n,u)-p(u,u) \geq r+\varepsilon \} \in \mathcal{I}$.
Let $M=r+a+\varepsilon$. Then if $ k \in A^c$, we can write 
\begin{equation*}
    \begin{split}
     p(x_k,u) & = |p(x_k,u)-p(u,u)+p(u,u)|  \\
              & \leq |p(x_k,u)-p(u,u)|+|p(u,u)| \\
              & < r+a+\varepsilon = M
               \end{split}
\end{equation*}
Hence $ k \in \{ n \in \mathbb N: p(x_n,u) < M \}$.
So, $ A^c \subset \{ n \in \mathbb N: p(x_n,u) < M \}$. 
Since $A \in \mathcal{I}$, so $\{ n \in \mathbb N : p(x_n,u) \geq M\} \in \mathcal{I}$. 
Hence $\{x_{n}\}$ is $\mathcal{I}$-bounded. 
\end{proof}

\begin{theorem}
    Let $\{ x_{{n}_{k}}\}$ be a subsequence of $\{x_{n}\}$ such that $\{ n_1, n_2,... \} \in \mathcal{F(I)}$, then $\mathcal{I}-LIM^{r}x_{n} \subseteq \mathcal{I}-LIM^{r}x_{n_{k}}$. 
\end{theorem}

\begin{proof}
Let $\varepsilon>0$ be arbitrary.
Let $x \in \mathcal{I}-LIM^{r}x_{n}$.
Then the set $A=\{ n \in \mathbb N : |p(x_n,x)-p(x,x)| \geq r+\varepsilon \} \in \mathcal{I}$.
Since the set $M=\{ n_1, n_2,... \} \in \mathcal{F(I)}$, so $A^c \cap M \neq \phi$. 
Let $ n_i \in A^c \cap M $. 
Then $ |p(x_{{n}_{i}},x)-p(x,x)| < r+\varepsilon $ 
i.e. $ n_i \in \{ n_k \in M: |p(x_{{n}_{k}},x)-p(x,x)| < r+\varepsilon \}=B$(say). 
So, $A^c \cap M \subset B$. This implies that $B^c \subset (A^c \cap M)^c=A \cup M^c$, i.e., $ \{ n_k \in M: |p(x_{{n}_{k}},x)-p(x,x)| \geq r+\varepsilon  \} \subset A \cup M^c$.
Since $(A \cup M^c) \in \mathcal{I}$, so we can conclude that 
 $ \{ n_k \in M: |p(x_{{n}_{k}},x)-p(x,x)| \geq r+\varepsilon \} \in \mathcal{I}$.  
Hence $ x \in \in \mathcal{I}-LIM^{r}x_{n_{k}}$. 
Therefore, $ \mathcal{I}-LIM^{r}x_{n} \subseteq \mathcal{I}-LIM^{r}x_{n_{k}}$.
\end{proof}

\begin{theorem}
Let  $\mathcal{I}$ be an admissible ideal and  $(X, p)$ be a partial metric space and let $\{a_{n}\}$ and $\{b_{n}\}$ be two sequences such that $p(a_{n}, b_{n}) \longrightarrow 0$ as $n \longrightarrow \infty$.
Then $\{a_{n}\}$ is rough-$\mathcal{I}$ convergent to $a$ and $p(a_{n},a_{n}) \longrightarrow 0$ as $n \longrightarrow \infty$ if and only if $\{b_{n}\}$ is rough-$\mathcal{I}$ convergent to $a$. 
\end{theorem}

\begin{proof}
 First suppose that $\{a_{n}\}$ is rough-$\mathcal{I}$ convergent to $a$ and $p(a_{n},b_{n}) \longrightarrow 0$ as $n \longrightarrow \infty$.
 Let $\varepsilon>0$.  
 Then the set $A=\{n \in \mathbb N: |p(a_n,a-p(a,a)| \geq r+\frac{\varepsilon}{3}\} \in \mathcal{I}$.
 Since $p(a_{n},b_{n}) \longrightarrow 0$ as $n \longrightarrow \infty$, so for this $\varepsilon>0$, there exists  $m_1 \in \mathbb N$ such that 
 \begin{equation} \label{a}
     p(a_{n},b_{n})\leq \frac{\varepsilon}{3}, \ \text{whenever} \ n \geq m_1.
 \end{equation}
   
Again, since $p(a_{n},a_{n}) \longrightarrow 0$ as $n \longrightarrow \infty$, there exists $ m_2 \in \mathbb N$ such that 
\begin{equation} \label{b}
    p(a_{n},a_{n})\leq \frac{\varepsilon}{3}, \ \text{whenever} \ n \geq m_2.
\end{equation}
Let $ m= \ max \{ m_1, m_2 \}$. Then equation (3.1) and (3.2) both hold for $n \geq m$. \\
Since $A \in \mathcal{I}$, $A^c \in \mathcal{F(I)}$.
Again, since $\mathcal{I}$ is an admissible ideal, $\{1,2,...m\} \in \mathcal{I}$.
So, $\{1,2,...m\}^c \in \mathcal{F(I)}$
and hence $A^c \cap \{1,2,...m\}^c \neq \phi$.
Also, if $n \in A^c \cap \{1,2,...m\}^c$ then we have
 \begin{equation*}
    \begin{split}
\ p(b_{n},a) &  \leq p(b_{n},a_{n})+p(a_{n},a)-p(a_{n},a_{n}) \\
\text{Therefore},\ p(b_{n},a) - p(a,a) & \leq p(b_{n},a_{n}) + p(a_{n},a) - p(a_{n},a_{n}) - p(a,a) \\
\text{This implies that} \ |p(b_{n},a) - p(a,a)|
    & \leq |p(b_{n},a_{n}) + p(a_{n},a) - p(a_{n},a_{n}) - p(a,a)| \\
    & \leq |p(b_{n},a_{n})| + |p(a_{n},a) - p(a,a)| + |p(a_{n},a_{n})|\\
    & < \frac{\varepsilon}{3} + (r+ \frac{\varepsilon}{3})+\frac{\varepsilon}{3} \\
    &  = r + \varepsilon 
    \end{split}
\end{equation*}
Hence $A^c \cap \{1,2,...m\}^c \subset \{n \in \mathbb N : |p(b_{n},a) - p(a,a)| < r+\varepsilon \}$.
So, $\{n \in \mathbb N : |p(b_{n},a) - p(a,a)| \geq r+\varepsilon \} \subset (A^c \cap \{1,2,...m\}^c)^c = A \cup \{1,2,...m\} \in \mathcal{I}$.
Therefore, $\{b_{n}\}$ is rough-$\mathcal{I}$ convergent to $a$. \\
Converse part is similar.
\end{proof}

\begin{theorem}
Let $\mathcal{I}$ be an admissible ideal and let $\{a_{n}\}$ and $\{b_{n}\}$ be two sequences in a partial metric space $(X, p)$ such that $p(a_{n},b_{n}) \longrightarrow 0$ as $ n \longrightarrow \infty$. 
If $\{a_{n}\}$ is rough-$\mathcal{I}$ convergent to $a$ of roughness degree $r$ and if $c$ is a positive number such that $ p(a_{n},a_{n}) \leq c$ for all $n$, then $\{b_{n}\}$ is rough-$\mathcal{I}$ convergent to $a$ of roughness degree $r+c$. 
Conversely, if $\{b_{n}\}$ is rough-$\mathcal{I}$ convergent to $b$ of roughness degree $r$ and $d$ is a positive number such that $p(b_{n},b_{n})\leq d$ for all $n$, then $\{a_{n}\}$ is rough-$\mathcal{I}$ convergent to $b$ of roughness degree $r+d$.      
\end{theorem}

\begin{proof}
The proof is parallel to the proof of the above theorem and so is omitted. 
\end{proof}

\begin{definition}
Let $(X, p)$ be a partial metric space. Then a point $c \in X$ is called an $\mathcal{I}$-cluster point of a sequence $\{x_{n}\}$ if for every $\varepsilon>0$, $\{n \in \mathbb N: |p(x_n,c)-p(c,c)| < \varepsilon \}) \notin \mathcal{I}$. The set of all 
$\mathcal{I}$-cluster points will be denoted by $\Lambda(\mathcal{I})$.
\end{definition}

\begin{theorem}
 Let $\{x_{n}\}$ be a sequence in a partial metric space $(X, p)$ and $p(x,x)=a$ for all $x \in X$, where $a$ be a real constant. If $c$ is a cluster point of $\{x_{n}\}$, then $\mathcal{I}-LIM^{r}x_{n}  \subset \overline{B^{p}_{r}}(c)$ for $r >0$.
\end{theorem}

\begin{proof}
Let $w \in \mathcal{I}-LIM^{r}x_{n}$ but $w \notin \overline{B^{p}_{r}}(c)$, where $\overline{B^{p}_{r}}(c)=\{y \in X: p(c,y) \leq p(c,c) + r\} =\{y \in X: p(c,y) \leq a + r\}$.
Then $a + r < p(c,w)$ and
let $\varepsilon^{'}= p(c,w) - (a+r)$.
So, $p(c,w)= \varepsilon^{'} + a + r$, where $\varepsilon^{'}>0$.
Choose $\varepsilon =\frac{\varepsilon^{'}}{2}$, so 
$p(c,w)= 2\varepsilon + a + r$.
Now, $B_{r+\varepsilon}(w) \cap B_{\varepsilon}(c) = \phi$.
For, if $B_{r+\varepsilon}(w) \cap B_{\varepsilon}(c) \neq \phi$, then there exists an element $y \in B_{r+\varepsilon}(w) \cap B_{\varepsilon}(c)$.
So, $p(w,y) < p(w,w) + r + \varepsilon = a + r + \varepsilon$  and $p(c,y) < p(c,c) + \varepsilon = a + \varepsilon$.
Now, \begin{equation*}
    \begin{split}
      p(c,w) & \leq p(c,y) + p(y,w) - p(y,y) \\
             & < \{a+\varepsilon \} + \{a+r+\varepsilon \}-a \\
             & = a+r+2\varepsilon = p(c,w), \ \text{a contradiction}.     
    \end{split}
\end{equation*}
Hence $B_{r+\varepsilon}(w) \cap B_{\varepsilon}(c) = \phi$.
Since $w \in \mathcal{I}-LIM^{r}x_{n}$, so for $\varepsilon > 0$, the set $A(\varepsilon)=\{n \in \mathbb{N}: |p(x_{n},w)-p(w,w)| \geq r+\varepsilon\} \in \mathcal{I}$.
Hence $(A(\varepsilon))^c \in \mathcal{F(I)}$.
Again, $c$ is a cluster point of $\{x_{n}\}$, so for $\varepsilon > 0$, the set $B(\varepsilon)=\{n \in \mathbb{N}: |p(x_{n},c)-p(c,c)| < \varepsilon \} \notin \mathcal{I}$.
Let $i \in (A(\varepsilon))^c \cap B(\varepsilon)$. 
Then 
$|p(x_{i},c)-p(c,c)|  < \varepsilon \Rightarrow p(x_{i},c) - p(c,c)  < \varepsilon
     \Rightarrow p(c,x_{i})   < p(c,c) + \varepsilon$.
So, $ x_{i} \in B_{\varepsilon}(c)$.
Again, 
 $|p(x_{i},w)-p(w,w)| < r+\varepsilon \Rightarrow p(x_{i},w)-p(w,w) < r+\varepsilon \Rightarrow p(w,x_{i}) < p(w,w) + r + \varepsilon$.
So, $ x_{i} \in B_{r+\varepsilon}(w)$. 
Hence $ x_{i} \in B_{r+\varepsilon}(w) \cap B_{\varepsilon}(c)$, which is a contradiction.
Therefore, $w \in \overline{B^{p}_{r}}(c)$.
\end{proof}

\subsection*{Acknowledgements}
The first author is thankful to The University of Burdwan for the grant of a senior research fellowship (State Funded) during the preparation of this article. First and second authors are also thankful to DST, Govt. of India, for providing the FIST project to the Department of Mathematics, B.U. \\

\end{document}